\documentclass[12pt, reqno]{amsart}

\usepackage{amssymb}

\usepackage{enumitem}
\usepackage{tikz}
\usepackage{tikz-cd}
\usetikzlibrary{matrix}
\usepackage{graphicx}
\usepackage{xcolor}
\usepackage{comment}
\makeatletter
\@namedef{subjclassname@2010}{%
  \textup{2020} Mathematics Subject Classification}
\makeatother

\usepackage[T1]{fontenc}

\usepackage{amsmath,amssymb,amsthm,amsfonts,latexsym,amscd,hyperref}
\allowdisplaybreaks
\hypersetup{colorlinks,linkcolor={blue},citecolor={blue},urlcolor={blue}}
\usepackage{xcolor}
\usepackage{mathtools}
\newtheorem{thm}{Theorem}[section]
\newtheorem{cor}[thm]{Corollary}
\newtheorem{lem}[thm]{Lemma}
\newtheorem{prop}[thm]{Proposition}

\newtheorem{question}[thm]{Question}
\theoremstyle{definition}
\newtheorem{defn}[thm]{Definition}
\theoremstyle{remark}
\newtheorem{rem}[thm]{Remark}
\theoremstyle{observation}

\theoremstyle{application}

\theoremstyle{note}

\newcommand{\vertiii}[1]{{\left\vert\kern-0.25ex\left\vert\kern-0.25ex\left\vert #1
    \right\vert\kern-0.25ex\right\vert\kern-0.25ex\right\vert}}
    
\newcommand{\bigslant}[2]{{\raisebox{.2em}{$#1$}\left/\raisebox{-.2em}{$#2$}\right.}}

\begin{document}


\baselineskip=17pt


\title[] { Position of $L(X, Y)$ in $Lip_0(X, Y)$}
\author[Anil Kumar Karn]{Anil Kumar Karn}
\address{Anil Kumar Karn, School of Mathematical Sciences, National Institute of Science Education and Research
Bhubaneswar, An OCC of Homi Bhabha National Institute, P.O. Jatni, Khurda, Odisha 752050,
India.}
\email{anilkarn@niser.ac.in}
\author[Arindam Mandal]{Arindam Mandal$^\dagger$}
\address{Arindam Mandal, School of Mathematical Sciences, National Institute of Science Education and Research
Bhubaneswar, An OCC of Homi Bhabha National Institute, P.O. Jatni, Khurda, Odisha 752050,
India.}
\email{arindam.mandal@niser.ac.in}
\thanks{Second author is supported by a research fellowship from the Department of Atomic Energy (DAE), Government of India.}
\thanks{$\dagger$ \tt{Corresponding author}}
\date{}

\begin{abstract} 
We prove that $L(X,Y)$ is complemented in $Lip_0(X, Y)$ (via a norm-one projection) provided that $Y$ is a dual space. Next, we introduce a vector-valued Lipschitz-free space $F_Y(X)$, a linear contraction $\beta_X^Y: F_Y(X) \to Y$ and prove that the quotient space ${Lip_{0}(X, Y)}/{L(X, Y) }$ is isometrically isomorphic to $L(\ker(\beta_X^Y), Y)$ whenever $Y$ is injective. We also consider a $Y$-valued duality pairing between $Lip_0(X, Y)$ and $F_Y(X)$ and obtain a necessary and sufficient condition for a Lipschitz map to be linear. As an application, we describe $\ker(\beta_{\mathbb{R}}^{\mathbb{R}})$ as the pre-dual of the quotient space $L^{\infty}(\mathbb{R})/\tilde{\mathbb{R}}$ where $\tilde{\mathbb{R}}$ is the set of all constant maps on $\mathbb{R}$. 
\end{abstract}

\subjclass[2023]{Primary 46B20; Secondary 46B03}
\keywords{Invariant mean,  Vector-valued duality, Generalized Lipschitz-free space.}

\maketitle
\pagestyle{myheadings}
\markboth{A. K. Karn and A. Mandal}{Projection onto Space of Linear Operators from Vector-Valued Lipschitz Function Space}
\section{Introduction and Preliminaries}

A mapping $f: X \to Y$ between two real Banach spaces $X$ and $Y$ is called Lipschitz if there exists $k > 0$ in $\mathbb{R}$ such that $\|f(x_1)- f(x_2)\| \le k \|x_1 - x_2\|$ for all $x_1, x_2 \in X$. The set of all such mappings is denoted by $Lip(X, Y)$, which forms a vector space with respect to pointwise addition and scalar multiplication. In the study of Lipschitz mappings, special attention is given to those mappings that send the $0$ to $0$. The set of such mappings is denoted by $Lip_0(X, Y)$. That is, $Lip_0(X, Y) = \lbrace f\in Lip(X, Y): f(0)=0\rbrace$. Further $Lip_{0}(X,Y)$ is a Banach space with the norm 
$$Lip(f)=\sup \left\lbrace \frac{\Vert f(x) - f(y) \Vert}{\|x-y\|}: x,y \in X, x \ne y \right\rbrace \mbox{\ for any \ } f \in Lip_{0}(X,Y).$$

The space $Lip_0(X, Y)$ is often regarded as a natural non-linear generalization of $L(X, Y)$, the space of bounded linear operators between $X$ and $Y$. It is known that $L(X, Y)$ is a subspace of $Lip_0(X, Y)$. The space $L(X, Y)$ carries the usual operator norm, 
$$\| T \| = \sup \lbrace \| T x \|: x \in X ~ \mbox{and} ~ \| x \| \le 1 \rbrace$$ 
for any $T \in L(X, Y)$. We have $\|T\|=Lip(T)$ for all $T \in L(X, Y)$ so that $L(X, Y)$ is a closed subspace of $Lip_0(X, Y)$. This leads to the following question:
\begin{question} \label{qns1}
	Is $L(X, Y)$ complemented in $Lip_0(X, Y)$?
\end{question}
The answer to this question is known to be affirmative in the case $Y = \mathbb{R}$. In fact, by a result due to Lindenstrauss (\cite[Theorem 2]{ONPIBS}) we know that there exists a (norm one) projection from $ Lip_{0}(X,\mathbb{R})$ onto $L(X,\mathbb{R})$ (see also \cite[Proposition 7.5]{GNFA}). In this paper, we extend the result and obtain an affirmative answer to Question \ref{qns1} when $Y$ is a dual space. We prove that there exists a surjective linear norm-one projection $P_X^Y : Lip_0(X, Y) \xrightarrow{} L(X, Y)$. This leads us to conclude that $L(X, Y)$ is complemented in $Lip_{0}(X, Y)$, whenever $Y$ is a dual space. More precisely, we prove that $Lip_0(X, Y)$ is equal to $L(X, Y) \oplus_{\infty} \ker (P_X^Y)$ as vector space and the two norms are equivalent.

Encouraged by first set of observations, we consider the next question:
\begin{question}\label{Q2}
    what is the quotient space: $\bigslant{Lip_{0}(X, Y)}{L(X, Y) }$? 
\end{question}
As $L(X, Y)$ is a closed subspace of $Lip_0(X, Y)$, $\bigslant{Lip_{0}(X, Y)}{L(X, Y) }$ is also a Banach space. We investigate whether it is possible to characterize it again as an operator space. We prepare to answer this question in the following way.

For each $x \in X$, we define $\delta_{x}^{Y} : Lip_{0}(X,Y) \xrightarrow{} Y$ by $\delta_{x}^{Y}(f)=f(x)$ for all $f \in Lip_{0}(X,Y)$. We show that $\delta_x^Y \in L(Lip_0(X, Y), Y)$ with $\Vert \delta_x^Y \Vert = \Vert x \Vert$ for all $x \in X$. This induces a Lipschitz mapping $\delta_X^Y: X \to L(Lip_0(X, Y), Y)$ given by $(x \mapsto \delta_x^Y)$. We consider the closed subspace  
$$F_{Y}(X) := \overline{span\{\delta_{x}^{Y} : x \in X\}}^{\|.\|}$$
of $L(Lip_{0}(X,Y),Y)$. (This is analogous to the Lipschitz-free space introduced by Godefroy and Kalton in \cite{LFBS}.) We then consider a bounded linear contraction $\beta_X^Y : F_Y(X) \xrightarrow{} X$ which we prove to be a left inverse of the Lipschitz map $\delta_X^Y: X\xrightarrow{} F_Y(X)$. This operator leads us to provide a partial answer to the Question \ref{Q2}. We prove that the space $\bigslant{Lip_{0}(X, Y)}{L(X, Y) }$ is isometrically isomorphic to $ L(\ker(\beta_X^Y), Y)$ when $Y$ is an injective Banach space.

Next, we examine the relation between $P_X^Y$ and $\beta_X^Y$. We note that $\ker(P_X^Y)$ is topologically isomorphic to $\bigslant{Lip_{0}(X, Y)}{L(X, Y) }$ which in turn can be proved to be isometrically isomorphic to $L(\ker(\beta_X^Y), Y)$, whenever $Y$ is a dual injective space. Both the results depend on the following observation: 
$$\{ f \in Lip_{0}(X, Y) : \gamma(f)=0 \ \ \forall \gamma \in \ker(\beta_X^Y)\} = L(X,Y).$$ 

We further prove that $L\left(\bigslant{F_Y(X)}{\ker(\beta_X^Y)}, Y\right)$ is isometrically isomorphic to  $L(X, Y)$, whenever $Y$ is a dual injective space. In fact, we prove more general versions of these results and discuss the embedding of $L(X, Y)$ into $L(F_{Y}(X), Y)$. 

At the end, we present an explicit description of the space $\ker(\beta_{\mathbb{R}}^{\mathbb{R}})$. We prove that it is isometrically isomorphic to $\{f \in L^{1}(\mathbb{R}) : \int_{\mathbb{R}}^{} f  d\mu=0\}$, where $\mu$ is the Lebesgue measure on $\mathbb{R}$. We deduce that 
$$\bigslant{L^{\infty}(\mathbb{R})}{\{ \mbox{set of all constant maps on}\ \mathbb{R} \}}$$ is isometrically isomorphic to dual of $\ker(\beta_{\mathbb{R}}^{\mathbb{R}})$. 

Now we will discuss the summary of the paper. In Section $2$, we have shown that there is a projection from the Lipschitz space to a linear space using an invariant mean. In Section $3$, we define the generalised Lipschitz free space $F_Y(X)$ of $X$ by considering a norm preserving non-linear map $\delta_X^Y: X \to L(Lip_0(X, Y), Y)$. We also define a projection-like map $\beta_X^Y : F_Y(X) \to X$ which is a left inverse of $\delta_X^Y$. We show that $F_Y(X)$ is lipeomorphic to the direct sum of $X$ and the kernel of $\beta_X^Y$.

In Section $4$, we prove that $Lip_0(X, Y)$ is isomorphic to $L(F_Y(X), Y)$. This leads us to consider a $Y$-valued duality between $Lip_0(X, Y)$ and $F_Y(X)$. One of the results here gives a necessary and sufficient condition for a Lipschitz map to be linear.

In Section $5$, we use the vector-valued duality from the previous section to define some new quotient spaces. In particular, we show that the Lipschitz space modulo the linear space is isomorphic to a certain operator space.

At the end, in Section $6$, we explicitly describe the kernel of 
$\beta_{\mathbb{R}}^{\mathbb{R}}$. Using the results developed in this paper, we give an alternative proof that $$\bigslant{L^{\infty}(\mathbb{R})}{\{ \mbox{set of all constant maps on}\ \mathbb{R} \}}$$
 is a dual space.




\section{Embedding of $L(X,Y)$ into $Lip_0(X, Y)$} \label{sec3}

Let $X$ and $Y$ be Banach spaces. We again recall that for $T \in L(X, Y)$, we have $\Vert T \Vert = Lip(T)$ so that $L(X, Y)$ is a closed subspace of $Lip_0(X, Y)$. In this section, we prove that $L(X, Y)$ is a complemented subspace of $Lip_0(X, Y)$ whenever $Y$ is a dual Banach space. In fact, we show that in this case, there exists a projection from $Lip_0(X, Y)$ onto $L(X, Y)$. We begin with the following observation. 

Let $f: X \to Y$ be a Lipschitz mapping and let $x \in X$. Consider $f_x: X \to Y$ given by $f_x(z) = f(x + z)$ for all $z \in X$. Then $f_x$ is also Lipschitz with $Lip(f_x) = Lip(f)$ for all $f \in Lip(X, Y)$ and $x \in X$. We use this translation to prove the next result.
\begin{prop}
	For any pair of Banach spaces $X$ and $Y$, $Lip_0(X, Y)$ is linearly isometrically isomorphic to a closed subspace of $Lip_0(X, \ell^{\infty}(X, Y))$.
\end{prop}
\begin{proof}
	We fix $f \in Lip_0(X, Y)$ and define $\phi_f : X \xrightarrow{} X^Y$ given by 
	$$\phi_f(x)=f_x-f \ \mbox{for all} \ x \in X.$$ 
	Let $x, z \in X$. Then 
	$$\|\left(\phi_f(x)\right)(z)\|= \|(f_x-f)(z)\|=\|f(x+z)-f(z)\|\leq Lip(f)\|x\|.$$
	Thus $\phi_f(x) \in \ell^{\infty}(X, Y)$ with $\|\phi_f(x)\|_{\infty} \leq Lip(f)\|x\|$. Now 
	\begin{eqnarray*}
		\sup\limits_{x \neq x'} \frac{\|\phi_f(x)-\phi_f(x')\|_{\infty}}{\|x-x'\|} &=&   \sup\limits_{x,x',z \in X; \ x \neq x'} \frac{\|\phi_f(x)(z)-\phi_f(x')(z)\|}{\|x-x'\|}\\
		&=&   \sup\limits_{x,x',z \in X; \ x \neq x'} \frac{\|f(x+z)-f(x'+z)\|}{\|(x+z)-(x'+z)\|} \\
		&=& \sup\limits_{x \neq x'} \frac{\|f(x)-f(x')\|_{\infty}}{\|x-x'\|} = Lip(f).
	\end{eqnarray*}
	Thus $\phi_f \in Lip_0(X, \ell^{\infty}(X, Y))$ with $Lip(\phi_f)=Lip(f)$. This induces a linear isometry $\Phi  : Lip_0(X, Y) \xrightarrow{} Lip_0(X, \ell^{\infty}(X, Y))$ given by $\Phi(f)=\phi_f$ for all $f \in Lip_0(X, Y)$. 
\end{proof} 
\begin{rem}
    For $f \in Lip(X, Y)$, $\phi_f \in Lip(X, \ell^{\infty}(X, Y))$. 
\end{rem} 
Now we recall the notion of an invariant mean. 
\begin{defn}\label{inv mean} \cite{GNFA}
	A left \emph{invariant mean} in $X$ is a linear functional $M$ on $\ell^{\infty}(X)$ (the space of all bounded maps from $X$ to $\mathbb{R}$) such that:
	\begin{enumerate}
		\item $M(\textbf{1})=1$, where $\textbf{1}(x) = 1$ for all $x \in X$. 
		\item $M(f)\geq 0$ for all $f \geq 0$.
		\item $M(f_x)=M(f)$ for all $x \in X$, where $f_x(z)= f(x+z)$ for all $z \in X$.
	\end{enumerate}
\end{defn}
Let $Y$ be a dual Banach space. It follows from \cite[p 417]{GNFA} that an invariant mean $M$ in $X$ induces a norm-one linear operator $\mathcal{M}_X^Y: \ell^{\infty}(X, Y) \to Y$ satisfying:
\begin{enumerate} 
    \item $\mathcal{M}_X^Y(f_x)=\mathcal{M}_X^Y(f)$ for all $f \in \ell^{\infty}(X, Y)$ and $x \in X$. 
    \item $\mathcal{M}_X^Y(\hat{y})=y$ for all $y \in Y$. Here $\hat{y} \in \ell^{\infty}(X, Y)$ is given by $\hat{y}(x) = y$ for all $x \in X$. 
\end{enumerate} 
We call $\mathcal{M}_X^Y$ a generalized invariant mean. It is easy to verify that $\mathcal{M}_X^{\mathbb{R}} = M$. 
\begin{rem}
	For $f \in Lip_0(X, Y)$ and $x,x',z \in X$, we have
	\begin{eqnarray*}
		\left(\left(\phi_{f_x}\right)(x')\right)(z) &=& \left((f_x)_x' -f_x\right)(z) = f(x+x'+z)-f(x+z)\\
		&=& \left(f_x' -f\right)(x+z)=\left(\phi_f(x')\right)_x(z).
	\end{eqnarray*}
	Thus $\left(\phi_{f_x}\right)(x')=\left(\phi_f(x')\right)_x \in \ell^{\infty}(X, Y)$ for all $x , x' \in X$. Now it follows that 
	$$\mathcal{M}_X^Y \left(\left(\phi_{f_x}\right)(x')\right)= \mathcal{M}_X^Y \left(\left(\phi_f(x')\right)_x\right) = \mathcal{M}_X^Y \left(\phi_f(x')\right)$$ 
	for all $f \in Lip_0(X, Y), x,x' \in X$.
\end{rem} 
\begin{thm}\label{proj}
	Let $X$ and $Y$ be Banach spaces with $Y$ be a dual space. Then $ L(X, Y)$ is complemented in $Lip_0(X, Y)$.
\end{thm}
\begin{proof}
	For $f \in Lip_0(X, )$ and $z \in X$, we set 
	$$P_{X}^{Y}(f)(z)= \mathcal{M}_X^Y (\phi_f(z))$$  
	where $\mathcal{M}_X^Y: \ell^{\infty}(X, Y) \xrightarrow{} Y$ is a vector-valued generalized invariant mean. Then $P_X^Y(f): X \to Y$. First, we show that $P_X^Y(f)$ is linear.
	
	
	Note that for $x, x' \in X$, we have 
	\begin{eqnarray*}
		\phi_f(x+x')&=& f_{x+x'}-f = f_{x+x'}-f_{x'}+f_{x'}-f\\
		&=& \left(f_x-f\right)_{x'} + (f_{x'}-f)\\
		&=& \left(\phi_f(x)\right)_{x'} + \phi_f(x').
	\end{eqnarray*}
	Thus as $\mathcal{M}_X^Y $ is translation invariant, we get 
	$$ P_{X}^{Y}(f)(x+x') = P_{X}^{Y}(f)(x) +P_{X}^{Y}(f)(x')$$ 
	for all $x, x' \in X$. Since $\mathcal{M}_X^Y$ and $\Phi$ are continuous, so is $P_X^Y(f)$. So, using a standard technique of real analysis, it is routine to prove that $P_{X}^{Y}(f)(\alpha x)= \alpha  P_{X}^{Y}(f)(x)$ for any $x \in X$ and $\alpha \in \mathbb{R}$ whence $P_X^Y(f) \in L(X, Y)$ for all $f \in Lip_0(X, Y)$. 
	Thus $P_X^Y: Lip_0(X, Y) \xrightarrow{} L(X, Y)$.
	
	Since $\mathcal{M}_X^Y$ and $\Phi$ are linear, so is $P_X^Y$. Also as 
	$$\|P_{X}^{Y}(f)(z)\|=\|\mathcal{M}_X^Y (\phi_f(z)) \| \leq Lip(f)\|z\|,$$
	we have $\|P_{X}^{Y}(f)\|\leq Lip(f)$ for all $f \in Lip_{0}(X, Y)$.
	
	Next, we prove that $P_{X}^{Y}(T)=T$ for all $T \in L(X, Y)$. Fix $T \in L(X, Y)$ and $z \in X$. Then for any $x \in X$, we have 
	$$\left(\phi_T(z)\right)(x)=T(x+z)-T(x) = T(z) \in Y,$$ 
	that is, $\phi_T(z) = \widehat{Tz}$.
	Thus 
	$$P_{X}^{Y}(T)(z)= \mathcal{M}_X^Y(\phi_T(z)) = \mathcal{M}_X^Y(\widehat{Tz})=Tz.$$ 
	Therefore, $P_X^Y(T) = T$ for all $T \in L(X, Y)$ so that $P_X^Y$ is a norm one linear projection from $Lip_0(X, Y)$ onto $L(X, Y)$. In other words, $L(X, Y)$ is complemented in $Lip_0(X, Y)$ whenever $Y$ is a dual Banach space. 
\end{proof}
\begin{cor}
	Let $X$ and $Y$ be Banach spaces with $Y$ be a dual space. Then $ Lip_0(X, Y)$ is topologically isomorphic to $L(X, Y) \oplus_1 \ker(P_X^Y)$.
\end{cor}
\begin{proof}

By Theorem \ref{proj}, $P_X^Y$ is a (norm one) projection with $P_X^Y(Lip_0(X, Y)) = L(X, Y)$. Thus $Lip_{0}(X, Y) = L(X, Y) \oplus \ker(P_{X}^{Y})$ as vector spaces. Here the linear isomorphism is given by $$f \in Lip_0(X, Y) \mapsto  (P_{X}^{Y}(f), (f - P_{X}^{Y}(f))) \in L(X, Y) \oplus \ker(P_{X}^{Y}).$$ 
Moreover, $\|.\|_1$ is equivalent to $Lip(.)$ on $Lip_0(X, Y)$, where 
$$\|f\|_1 = \|P_{X}^{Y}(f)\| + Lip(f - P_{X}^{Y}(f))$$ 
for all $f \in Lip_0(X, Y)$. In fact, 
$$Lip(f) = \Vert P_X^Y(f) \Vert + Lip(f - P_{X}^{Y}(f)) \leq Lip(f) + 2 \Vert P_{X}^{Y}(f) \Vert \leq 3 Lip(f),$$  
if $f \in Lip_{0}(X, Y)$.
\end{proof}
\begin{rem} \label{ker p with quotient}
    We note that $\ker(P_X^Y)$ is topologically isomorphic to $\bigslant{Lip_{0}(X, Y)}{L(X, Y) }$. In fact, $Id_{Lip_0(X,Y)} - P_X^Y$ is a bounded linear surjective projection form $Lip_0(X,Y)$ onto  $\ker(P_X^Y)$ with kernel $L(X, Y)$. 
\end{rem} 
 \section{Generalized Lipschitz-Free space} \label{sec4}

If $X$ and $Y$ are Banach spaces, then $L(X, Y)$ is a closed subspace of the Banach space $Lip_0(X, Y)$ and therefore $ \bigslant{Lip_0(X, Y)}{L(X, Y)}$ is also a Banach space. In the next section, we shall obtain a condition on $Y$ under which the quotient space is seen as an operator space. To facilitate this discussion, in this section we describe a generalization of Lipschitz free space. We may recall that the study of the linear structure of Lipschitz-free spaces has been an active field of study. See, for example, \cite{OTSOLFS}, \cite{LA2}, \cite{ASOLFBS}, \cite{LFbook}, \cite{FSOCCMS}, \cite{ASPFLFSOCMS}. For the sake of completeness, we have included the proofs, though some of these are a modification of the classical case, 
 \begin{prop}
    Let $X$ and $Y$ be Banach spaces. 
    \begin{enumerate}
        \item The map $\delta_{x}^{Y} : Lip_{0}(X,Y) \xrightarrow{} Y$ defined by $\delta_{x}^{Y}(f)=f(x)$ for all $f \in Lip_{0}(X,Y)$ is a bounded linear map with $\|\delta_{x}^{Y}\|=\|x\|$ for all $x \in X$.
        \item For any $x_{1},x_{2} \in X$; $\|\delta_{x_{1}}^{Y} - \delta_{x_{2}}^{Y}\|= \|x_{1}-x_{2}\|$.
    \end{enumerate}
    \end{prop}
\begin{proof}
Let $x,x_{1},x_{2} \in X$.
\begin{enumerate}
    \item It is routine to check that $\delta_{x}^{Y}$ is linear. Let $f \in Lip_{0}(X,Y)$. Then $$\|\delta_{x}^{Y}(f)\|= \|f(x)\| \leq Lip(f)\|x\|.$$ 
    Thus $\|\delta_{x}^{Y}\|\leq \|x\|$. Next, fix $y_0 \in Y$ with $\Vert y_0 \Vert = 1$ and consider $g:X \xrightarrow{} Y$ defined as $g(t) = \|t\|y_{0}$ for all $t \in X$. Then $g \in B_{Lip_{0}(X,Y)}$ (closed unit ball in $Lip_0(X, Y)$). Now $\|\delta_{x}^{Y}(g)\|= \|g(x)\|=\|x\|$ so that $\Vert \delta_x^Y \Vert = \|x\|$.  
 \item Again, 
 $$\|(\delta_{x_{1}}^{Y} - \delta_{x_{2}}^{Y})(f)\|= \|f(x_{1})-f(x_{2})\| \leq Lip(f)\|x_{1}-x_{2}\|.$$ 
 Thus $\|\delta_{x_{1}}^{Y} - \delta_{x_{2}}^{Y}\|\leq \|x_{1}-x_{2}\|$. Again, fix $y_0 \in Y$ with $\Vert y_0 \Vert = 1$ and consider the map $h: X \to Y$  defined as $h(t) = \|x_{2}\|y_{0}-\|t-x_{2}\|y_{0}$ for all $t \in X$. Then $h\in B_{Lip_{0}(X,Y)}$ and we have 
 $$\|(\delta_{x_{1}}^{Y} - \delta_{x_{2}}^{Y})(h)\| = \Vert h(x_1) - h(x_2) \Vert = \|x_{1}-x_{2}\|.$$
   \end{enumerate}
This completes the proof.
\end{proof}
\begin{rem} \label{rem1}
    For Banach spaces $X$ and $Y$, the map $\delta_{X}^{Y} : X \xrightarrow{} L(Lip_{0}(X,Y),Y)$ defined as $\delta_{X}^{Y}(x) = \delta_{x}^{Y}$, is a Lipschitz map with $\delta_X^Y(0) = 0$. Thus $\delta_{X}^{Y} \in Lip_{0}(X,L(Lip_{0}(X,Y),Y))$. Also, it is a non-linear isometry.
\end{rem}
Now we consider $F_{Y}(X) := \overline{span\{\delta_{x}^{Y} : x \in X\}}^{\|.\|}$, a closed subspace of $L(Lip_{0}(X,Y),Y)$. Then $\delta_X^Y: X \to F_Y(X, Y)$ is a (non-linear) isometry. 
The following result describes that every element of the generalized Lipschitz free space can be represented as an infinite series. This fact will be used throughout the article without further reference.
\begin{prop}
	Let $\mu \in F_Y(X)$. Then for every $\epsilon >0$, there exist sequences $(a_n)$ in $\mathbb{R}$ and $(x_n)$ in $X$ such that $\mu = \sum\limits_{n=1}^{\infty}a_n \delta_{x_n}^Y$ and $\sum\limits_{n=1}^{\infty}|a_n| \|\delta_{x_n}^Y\| = \sum\limits_{n=1}^{\infty}|a_n| \|x_n\| < \|\mu\|+\epsilon$.
\end{prop}
\begin{proof}
	Let $\mu \in F_Y(X)$ and $\epsilon > 0$. By Lemma \cite[3.100]{BST} there exists a sequence $(\mu_n)$ in $span\left\{\delta_x^Y : x \in X\right\}$ such that $\mu = \sum\limits_{n=1}^{\infty} \mu_n$ and $$\sum\limits_{n=1}^{\infty} \|\mu_n\| < \|\mu\| + \frac{\epsilon}{2}.$$
	For each $n \in \mathbb{N}$, we have $\mu_n = \sum\limits_{i=1}^{I_n} a_i^n \delta_{x_i^n}^Y$ such that $\sum\limits_{i=1}^{I_n} |a_i^n| \|x_i^n\| < \|\mu_n\| + \frac{\epsilon}{2^{n+1}}$, for some $I_n \in \mathbb{N}$.
	We re-index the sequences $(a_i^n)_{n,i}, (x_i^n)_{n,i}$ as $(a_j)_{j=1}^{\infty}, (x_j)_{j=1}^{\infty}$, respectively. Then
	$$\sum\limits_{j=1}^{\infty} |a_j| \|x_j\| = \sum\limits_{n,i}^{}|a_i^n| \|x_i^n\|= \sum\limits_{n=1}^{\infty}\sum\limits_{i=1}^{I_n}|a_i^n| \|x_i^n\| \leq \sum\limits_{n=1}^{\infty} \left(\|\mu_n\| + \frac{\epsilon}{2^{n+1}}\right) \le \|\mu\| + \epsilon.$$
	Therefore, the series $\sum\limits_{j=1}^{\infty} a_j \delta_{x_j}^Y$ is absolutely convergent in the Banach space $F_Y(X)$ and hence $\sum\limits_{j=1}^{\infty} a_j \delta_{x_j}^Y \in F_Y(X)$. We show that $\mu = \sum\limits_{j=1}^{\infty} a_j \delta_{x_j}^Y$.
	
	Let $t >0$. Since $\mu = \sum\limits_{n=1}^{\infty} \mu_n$ and $\sum\limits_{n=1}^{\infty} \|\mu_n\| < \|\mu\| + \frac{\epsilon}{2}$, there exist $N_1, N_2, N_3 \in \mathbb{N}$ such that $\|\mu-\sum\limits_{j=1}^{m} \mu_n\| <t$ for all $m \geq N_1$; $\|\mu_m\| <t$ for all $m \geq N_2$ and $\frac{\epsilon}{2^{m}} < t$ for all $m \geq N_3$. Put $N = max\{N_1, N_2, N_3\}$. Then $max\left\{\|\mu-\sum\limits_{j=1}^{m} \mu_n\|, \|\mu_m\|, \frac{\epsilon}{2^{m}} \right\} < t$ for all $m \geq N$. Now for all $n > \sum\limits_{k=1}^{N} I_k$
	\begin{eqnarray*}
		\left\|\mu-\sum\limits_{j=1}^{n} a_j \delta_{x_j}^Y\right\| &\leq& \left\|\mu-\sum\limits_{j=1}^{m-1} \mu_j\right\| + \sum\limits_{i=1}^{I_m} |a_i^m| \|x_i^m\|  \\
		&\leq& \left\|\mu-\sum\limits_{j=1}^{m-1} \mu_j\right\| + \|\mu_m\|+ \frac{\epsilon}{2^m} < 3t. 
	\end{eqnarray*}
	This completes the proof.
\end{proof}
In the next result, we describe a left inverse of $\delta_{X}^{Y}$. 
\begin{prop}
     Let $X$ and $Y$ be Banach spaces. Then there exists a contractive linear map $\beta_X^Y: F_Y(X) \to X$ such that $\beta_X^Y \delta_X^Y = Id_X$. 
\end{prop}
\begin{proof}
	Consider $\beta_{X}^{Y}(\sum\limits_{i=1}^{n}\alpha_{i} \delta_{x_{i}}^{Y})= \sum\limits_{i=1}^{n}\alpha_{i} x_{i}$ for any finite collection $\alpha_{1}, \alpha_{2},...,\alpha_{n} \in \mathbb{R}$ and $x_{1}, x_{2},...,x_{n} \in X$. Then $\beta_{X}^{Y} : span\{\delta_{x}^{Y} : x \in X\} \xrightarrow{} X$ is a linear map with $\beta_{X}^{Y}\delta_{X}^{Y} = Id_{X}$. Also 
\begin{eqnarray*}
   \| \beta_{X}^{Y}(\sum\limits_{i=1}^{n}\alpha_{i} \delta_{x_{i}}^{Y})\| &=& \|\sum\limits_{i=1}^{n}\alpha_{i} x_{i}\|\\
   &=& \sup \big\{ |x^{\ast}(\sum\limits_{i=1}^{n}\alpha_{i} x_{i})| : x^{\ast} \in B_{X^{\ast}}\big\}\\
   &=& \sup \big\{ |x^{\ast}(\sum\limits_{i=1}^{n}\alpha_{i} x_{i}) y_{0}| : x^{\ast} \in B_{X^{\ast}}, y_{0} \in Y, \|y_{0}\|=1\big\}\\
    &=& \sup \big\{ |\sum\limits_{i=1}^{n}\alpha_{i} \ x^{\ast} (x_{i}) y_{0}| : x^{\ast} \in B_{X^{\ast}}, y_{0} \in Y, \|y_{0}\|=1\big\}\\
    &\leq& \sup \big\{ |\sum\limits_{i=1}^{n}\alpha_{i} \ f (x_{i}) y_{0}| : f \in B_{Lip_0(X, \mathbb{R})}, y_{0} \in Y, \|y_{0}\|=1\big\}\\
    &\leq& \|\sum\limits_{i=1}^{n}\alpha_{i} \ \delta_{x_{i}}^{Y}\|.
\end{eqnarray*}
Thus $\beta_{X}^{Y}$ is a bounded linear map with $\|\beta_{X}^{Y}\| \leq 1$. Hence, it can be extended uniquely to a contractive linear map from $F_Y(X)$ into $X$ which, we again denote as $\beta_{X}^{Y}$.
\end{proof}
\begin{prop}
    The adjoint map $(\beta_{X}^{Y})^{\ast}: X^{\ast}  \xrightarrow{} F_{Y}(X)^{\ast} $ is an isometry.
\end{prop}
\begin{proof}
    Let $x^{\ast} \in X^{\ast}$. Since $\Vert \beta_X^Y \Vert \le 1$, we get 
    $$\|(\beta_{X}^{Y})^{\ast}(x^{\ast})\| \leq \| (\beta_X^Y)^{\ast} \|x^{\ast}\| \le \| x^{\ast} \|.$$ 
    Also, 
    \begin{eqnarray*}
    \|(\beta_{X}^{Y})^{\ast}(x^{\ast})\|  &=& \sup\limits_{\gamma \in B_{F_{Y}(X)}} \|(\beta_{X}^{Y})^{\ast}(x^{\ast})(\gamma)\|\\
    &\geq&  \sup\limits_{x \in B_{X}} \|(\beta_{X}^{Y})^{\ast}(x^{\ast})(\delta_{x}^{Y})\|\\
    &=& \sup\limits_{x \in B_{X}} \|x^{\ast} \circ \beta_{X}^{Y}(\delta_{x}^{Y})\|\\
    &=&\sup\limits_{x \in B_{X}} \|x^{\ast}(x)\|=\|x^{\ast}\|  
    \end{eqnarray*}
    Thus $\|(\beta_{X}^{Y})^{\ast}(x^{\ast})\| = \| x^{\ast} \|$ for all $x^{\ast} \in X^{\ast}$. 
\end{proof}
Now we discuss a lipeomorphic decomposition of $F_{Y}(X)$ as $X \oplus \ker(\beta_{X}^{Y})$. (By a \emph{lipeomorphism}, we mean a bijection via a Lipschitz map whose inverse is also Lipschitz.) 
\begin{prop}
    Let $X$ and $Y$ be any two Banach spaces.  
    Then $F_{Y}(X)$ is lipeomorphic to $X \oplus_1 \ker(\beta_{X}^{Y})$. 
\end{prop}
\begin{proof}
  For $\mu \in \ker(\beta_{X}^{Y})$ and $x \in X$, we define 
  $$\eta_{X}^{Y} : X \oplus_1 \ker(\beta_{X}^{Y}) \xrightarrow{} F_{Y}(X)$$ 
  given by  $\eta_{X}^{Y}(x,\mu)=\delta_{x}^{Y} + \mu$.
  
     Let $\mu_{1}, \mu_{2} \in \ker(\beta_{X}^{Y})$, $\gamma_1, \gamma_2 \in F_Y(X)$ and $x_{1}, x_{2} \in X$ then 
    \begin{eqnarray*}
         \|\eta_{X}^{Y}(x_{1},\mu_{1})-\eta_{X}^{Y}(x_{2},\mu_{2})\|&=& \|\mu_{1}-\mu_{2}+\delta_{x_{1}}^{Y}-\delta_{x_{2}}^{Y}\|\\ &\leq& \|\mu_{1}-\mu_{2}\|+\|\delta_{x_{1}}^{Y}-\delta_{x_{2}}^{Y}\| \\
         &=& \|\mu_{1}-\mu_{2}\|+\|x_{1} -x_{2}\|\\
         &=& \|(x_{1}-x_{2},\mu_{1}-\mu_{2})\|_{1}.
    \end{eqnarray*}
    Thus $\eta_{X}^{Y}$ is a Lipschitz map with $Lip(\eta_X^Y) \le 1$. 
    Now we show that $\eta_X^Y$ is injective. Suppose $\delta_{x_1}^{Y} + \mu_1=\delta_{x_2}^{Y} + \mu_2$. Then appling $\beta_X^Y$ on both sides we have $x_1=x_2$ and hence $\mu_1= \mu_2$. Therefore $\eta_X^Y$ is one-one.
    
    Again for $\gamma \in F_Y(X)$, put $\alpha=(\beta_{X}^{Y}(\gamma),\gamma-\delta_{X}^{Y}\beta_{X}^{Y}(\gamma))$. Then $\alpha \in X \oplus_1 \ker(\beta_{X}^{Y})$ and  $$\eta_{X}^{Y} (\alpha) = \eta_{X}^{Y}(\beta_{X}^{Y}(\gamma), \gamma-\delta_{X}^{Y}\beta_{X}^{Y}(\gamma))= \gamma-\delta_{X}^{Y}\beta_{X}^{Y}(\gamma)+\delta_{X}^{Y}\beta_{X}^{Y}(\gamma)=\gamma.$$
    Thus $\eta_X^Y$ is surjective also and $(\eta_{X}^{Y})^{-1}$ exists.
    
    Again
    \begin{eqnarray*}
         \|(\eta_{X}^{Y})^{-1}(\gamma_{1})-(\eta_{X}^{Y})^{-1}(\gamma_{2})\|&=&\|\beta_{X}^{Y}(\gamma_{1})-\beta_{X}^{Y}(\gamma_{2}),\gamma_{1}-\gamma_{2}+\delta_{X}^{Y}\beta_{X}^{Y}(\gamma_{2})-\delta_{X}^{Y}\beta_{X}^{Y}(\gamma_{1})\|_{1}\\
         &=&\|(\gamma_{1}-\gamma_{2}+\delta_{X}^{Y}\beta_{X}^{Y}(\mu_{2})-\delta_{X}^{Y}\beta_{X}^{Y}(\gamma_{1})\|+\|\beta_{X}^{Y}(\gamma_{1})-\beta_{X}^{Y}(\gamma_{2}))\| \\
          &\leq&\|\gamma_{1}-\gamma_{2}\|+\|\delta_{X}^{Y}\beta_{X}^{Y}(\gamma_{2})-\delta_{X}^{Y}\beta_{X}^{Y}(\gamma_{1})\|+\|\gamma_{1}-\gamma_{2}\| \\
          &\leq& \|\gamma_{1}-\gamma_{2}\|+Lip(\delta_{X}^{Y}\beta_{X}^{Y})\|\gamma_{1}-\gamma_{2}\| + \|\gamma_{1}-\gamma_{2}\|\\
         &\leq& 3\|\gamma_{1}-\gamma_{2}\|.
     \end{eqnarray*}
    Thus $(\eta_{X}^{Y})^{-1}$ is also a Lipschitz map with $Lip((\eta_X^Y)^{-1}) \le 3$. 
    This completes the proof. 
\end{proof}

 \section{{\bf A vector valued duality}} \label{section5}
 
\begin{thm} \label{thm 1}
     Let $X$ and $Y$ be Banach spaces. Then $(Lip_{0}(X,Y),Lip(\cdot))$ is isometrically isomorphic to $L(F_{Y}(X),Y)$. 
\end{thm}
\begin{proof}
    Define $\Delta_{X}^{Y}: (L(F_{Y}(X),Y),\|\cdot\|) \xrightarrow{} (Lip_{0}(X,Y),Lip(\cdot))$ given by $\Delta_{X}^{Y}(T)= T \circ \delta_{X}^{Y}$ for all $T \in L(F_{Y}(X),Y)$. Then $\Delta_{X}^{Y}$ is linear. Also, for any $T \in L(F_{Y}(X),Y)$, we have 
    $$Lip(\Delta_{X}^{Y}(T))= Lip(T \circ \delta_{X}^{Y}) \leq \|T\| Lip(\delta_{X}^{Y}) \leq \|T\|.$$ Thus, $\Delta_{X}^{Y}$ is a contraction. We show that it is surjective. Let $g \in Lip_{0}(X,Y)$. Set $\hat{g}: F_{Y}(X) \to Y$ given by $\hat{g}(\gamma) = \gamma(g) ~ \mbox{for all } \gamma \in F_Y(X)$. Then $\hat{g} \in L(F_Y(X), Y)$ with $\|\hat{g}\| \leq Lip(g)$. Further 
	$$\|\hat{g}\| \geq \sup\limits_{x_1 \ne x_2} \left\|\frac{\delta_{x_1}^Y-\delta_{x_2}^Y}{\|x_1-x_2\|} (g)\right\|= Lip(g).$$ Thus $\|\hat{g}\|= Lip(g)$. Also $\Delta_X^Y (\hat{g})= \hat{g} \circ \delta_X^Y= g$ so that $\Delta_X^Y$ is surjective. Now by open mapping theorem $(\Delta_X^Y )^{-1}$ is a bounded linear map. In fact, 
	\begin{eqnarray*}
		\|(\Delta_X^Y )^{-1}(f)\| &=& \sup \{ \Vert (\Delta_X^Y )^{-1}(f)(\mu): \Vert \mu \Vert \le 1 \} \\ 
		&=& \sup \{ \Vert \mu(f) \Vert: \Vert \mu \Vert \le 1 \} \\ 
		&\le& Lip(f)
	\end{eqnarray*} 
	for all $f \in Lip_0(X, Y)$. Hence $\Delta_X^Y$ is a surjective linear isometry.
    \end{proof}
    \begin{rem} \label{six rem}
    \begin{enumerate} 
    	Let $X$ and $Y$ be Banach spaces.
        \item For each $f \in Lip_{0}(X,Y)$ there exists unique linear map $T := \nabla_{X}^{Y}(f) \in L(X, Y)$ such that $f = T \circ \delta_{X}^{Y}$. 
        \item Let $T \in L(X, Y)$. Then $\nabla_{X}^{Y}(T) = T \circ \beta_{X}^{Y}$. Thus $$\nabla_{X}^{Y}(L(X, Y))= \{T \circ \beta_{X}^{Y}: T \in L(X,Y)\}$$
        is a closed subspace of $L(F_Y(X, Y), Y)$. 
        \item $\nabla_{X}^{Y}(T \circ \delta_{X}^{Y}) = T$ for any $T \in L(X,Y)$. 
        \item For $Y = \mathbb{R}$, we have $(\beta_{X}^{\mathbb{R}})^{\ast} = {\nabla_{X}^{\mathbb{R}}}\big|_{L(X, \mathbb{R})}$. Thus the image of $(\beta_{X}^{\mathbb{R}})^{\ast}$ is $\nabla_{X}^{\mathbb{R}}\left(L(X, \mathbb{R})\right)$.
        \item When $Y = \mathbb{R}$, $F_{\mathbb{R}}(X)$ appears as $F(X)$ in \cite{LCO} and is called a Lipschitz free space. In the same paper, $\delta_{X}^{\mathbb{R}}$ and $\delta_{x}^{\mathbb{R}}$ are denoted as $\delta_{X}$ and $\delta_{x}$ respectively for all $x \in X$ and $\nabla_{X}^{\mathbb{R}}$ is denoted as $Q_{X}$. Moreover, Theorem \ref{thm 1} is a generalization of \cite[Lemma~1.1]{LCO}. However, the two proofs are different. 
    \end{enumerate}
    \end{rem}
Theorem \ref{thm 1} encourages us to consider the following `vector-valued' duality.

\begin{prop}
    For any two Banach spaces $X$ and $Y$, define
    $$\langle \cdot, \cdot \rangle_{Y} : Lip_{0}(X, Y) \times F_Y(X) \xrightarrow{} Y$$ given by $$\langle f,\mu \rangle_{Y}:= \mu(f)$$ for all $f \in Lip_{0}(X, Y)$ and $\mu \in F_Y(X)$. Then $\langle \cdot, \cdot \rangle_{Y}$ is bilinear and non-degenerate. 
\end{prop}
\begin{proof}
    It is routine to check that $\langle \cdot, \cdot \rangle_{Y}$ is bilinear. 
    
    Let $f \in Lip_{0}(X, Y)$ be such that $\mu(f)=0$ for all $\mu \in F_Y(X)$. Then $\delta_{x}^{Y}(f)=0$ for all $x \in X$. Thus $f =0$. Conversely, consider $\mu \in F_Y(X)$ with $\mu(f)=0$ for all $f \in Lip_{0}(X, Y)$. Since $F_Y(X)$ is a subspace of $L(Lip_0(X, Y), Y)$, we have $\mu =0$.
\end{proof}
For $\mathcal{A} \subset Lip_{0}(X, Y)$, we define $\mathcal{A}^{\diamondsuit}:=\{\gamma \in F_Y(X): \gamma(A)=0 \ \ \forall A \in \mathcal{A}\}$. Then $\mathcal{A}^{\diamondsuit}$ is a closed subspace of $F_Y(X)$.

Similarly, for $\mathcal{D} \subset F_Y(X)$,  $~^{\diamondsuit}\mathcal{D}:= \{ f \in Lip_{0}(X, Y) : \gamma(f)=0 \ \ \forall \gamma \in \mathcal{D}\}$ is a closed subspace of $Lip_{0}(X, Y)$.
\begin{lem}\label{eqn2}
	Let $X$ and $Y$ be Banach spaces. Then $L(X, Y) = ~^\diamondsuit\ker(\beta_X^Y)$.  
\end{lem}
\begin{proof}
	 Let $f \in Lip_{0}(X, Y)$ such that $\mu(f)=0$ for all $\mu \in \ker(\beta_{X}^{Y})$.
	We show that $f$ is linear. Let $x_{1},x_{2} \in X$ and $r \in \mathbb{R}$. Then  $- r \delta_{x_1}^{Y} - \delta_{x_2}^Y + \delta_{r x_1 + x_2}^{Y} \in \ker(\beta_{X}^{Y})$. Thus 
	$$(- r \delta_{x_1}^{Y} - \delta_{x_2}^Y + \delta_{r x_1 + x_2}^{Y})(f) = 0$$ 
	which amounts to saying 
	$$- r f(x_1) - f(x_2) + f( r x_1 + x_2) = 0.$$ 
	Thus $f( r x_1 + x_2) = r f(x_1) + f(x_2)$ so that $f$ is linear. Now, being Lipschitz, $f$ is uniformly continuous in $X$ so that $f \in L(X, Y)$. 
	
	Conversely, assume that $T \in L(X, Y)$ and let $\mu \in \ker(\beta_{X}^{Y})$. Then $\mu = \sum\limits_{n=1}^{\infty} a_{n} \delta_{x_{n}}^{Y}$,  for some $(x_{n}) \subset X$ and $(a_{n}) \subset \mathbb{R}$ with $\sum\limits_{n=1}^{\infty} a_{n} x_{n} =0$. Therefore 
	$$\mu(T) = \sum\limits_{n=1}^{\infty} a_{n} \delta_{x_{n}}^Y (T) = \sum\limits_{n=1}^{\infty} a_{n} T(x_{n}) = T \left(\sum\limits_{n=1}^{\infty} a_{n} x_{n}\right) = 0.$$ 
\end{proof}
\begin{lem}
    Let $X$ and $Y$ be Banach spaces. Then $L(X, Y)^\diamondsuit = \ker(\beta_X^Y)$.
\end{lem}
\begin{proof}
    Let $\mu \in L(X, Y)^\diamondsuit$. Then there exist $(x_n) \subset X$ and $(a_n) \subset \mathbb{R}$ such that $\mu = \sum\limits_{n=1}^{\infty} a_n \delta_{x_n}^Y$ with $\mu(T)=0$ for all $T \in L(X, Y)$. That is 
    $$T\left(\sum\limits_{n=1}^{\infty} a_n x_n\right)=\sum\limits_{n=1}^{\infty} a_n T(x_n)=\left(\sum\limits_{n=1}^{\infty} a_n \delta_{x_n}^Y\right)(T)=0$$
    for all $T \in L(X, Y)$. Therefore $\sum\limits_{n=1}^{\infty} a_n x_n=0$ and hence $\mu \in \ker(\beta_X^Y)$.

    Again suppose $\mu \in \ker(\beta_X^Y)$. Consequently, we have 
$\mu(T)=0$ for all $T \in L(X, Y)$, analogous to the argument used in the proof of the converse direction of the previous result.
\end{proof}
\begin{rem}
    From the above two results, for any two Banach spaces $X$ and $Y$ we have $\left(~^\diamondsuit\ker(\beta_X^Y)\right)^\diamondsuit = \ker(\beta_X^Y)$ and $~^{\diamondsuit}\left(L(X, Y)^{\diamondsuit}\right) = L(X, Y)$.
\end{rem}

\section{Some quotient spaces}

\begin{prop} \label{quotient of lip}
	Let $\mathcal{D}$ is a closed subspace of $F_{Y}(X)$, where $Y$ be an injective Banach space; then $\bigslant{Lip_{0}(X, Y)}{^{\diamondsuit}\mathcal{D} }$ is linear isometrically isomorphic to $  L(\mathcal{D}, Y)$. 
\end{prop}
\begin{proof}
	For a fix $f \in Lip_0(X, Y)$,  we define $\theta_f : \mathcal{D} \to Y$ given by $\theta_f(\gamma) = \gamma(f)$, for all $\gamma \in \mathcal{D}$. Then $\theta_f$ is linear. Also, for any $\gamma \in \mathcal{D}$, we have 
		$$\|\theta_f(\gamma)\|=\|\gamma(f)\|\leq \|\gamma\| Lip(f).$$ 
	Thus $ \theta_f \in L(\mathcal{D}, Y)$  for any $f \in Lip_0(X, Y)$. 
	
	We define $\Theta : Lip_0(X, Y) \to L(\mathcal{D}, Y)$ given by $\Theta(f)= \theta_f$. It is routine to check that $\Theta$ is linear. We show that $\Theta$ is surjective also. Let $T \in L(\mathcal{D}, Y)$.  Since $Y$ is injective there exists $\widetilde{T} \in L(F_{Y}(X), Y)$ such that $\widetilde{T}\big|_{\mathcal{D}}=T$ with $\|\widetilde{T}\|=\|T\|$. Thus $\widetilde{T} \circ \delta_X^Y \in Lip_0(X, Y)$. Now, for $\gamma = \sum\limits_{n=1}^{\infty} a_{n} \delta_{x_{n}}^{Y} \in \mathcal{D}$, we get  $$\Theta(\widetilde{T} \circ \delta_{X}^{Y} )(\gamma)= \gamma \left(\widetilde{T} \circ \delta_{X}^{Y}\right)= \sum\limits_{n=1}^{\infty} a_{n} \widetilde{T} \circ \delta_{X}^{Y}(x_{n}) = \widetilde{T} \left(\sum\limits_{n=1}^{\infty} a_{n} \delta_{x_{n}}^{Y}\right)=T(\gamma).$$ Thus $\Theta$ is surjective.
	
	Also, 
	\begin{eqnarray*}
		\ker(\Theta)&=&\left\{f\in Lip_{0}(X, Y): \theta_f=0 \right\}\\
		&=&\left\{f\in Lip_{0}(X, Y): 0 = \theta_f (\gamma)= \gamma(f) ~ \mbox{for all}~ \gamma \in \mathcal{D}\right\}\\
		&=& \ ^{\diamondsuit}\mathcal{D}.
	\end{eqnarray*} 
	So by fundamental theorem of linear algebra $\Theta$ induce a linear isomorphism (which we again denote by $\Theta$): 
	$$\Theta : \bigslant{Lip_{0}(X, Y)}{^{\diamondsuit}\mathcal{D} } \to L(\mathcal{D}, Y)$$
	given by $\Theta(f+ \ ^{\diamondsuit}\mathcal{D})(\gamma)=\gamma(f)$ for all $\gamma \in \mathcal{D}$.
	
	Now, for any $f \in Lip_0(X, Y)$ and $\gamma \in \mathcal{D}$
	\begin{eqnarray*} \label{cont of theta1}
		\|\Theta(f+ \ ^{\diamondsuit}\mathcal{D})(\gamma)\| &=& \|\gamma(f)\| \\ 
		&=& \inf\limits_{g \in \ ^{\diamondsuit}\mathcal{D} }\|\gamma(f+g)\| \\ 
		&\le& \|\gamma\|\inf\limits_{g \in \ ^{\diamondsuit}\mathcal{D} }Lip(f+g) \\ 
		&=& \|\gamma\| \|f+ \ ^{\diamondsuit}\mathcal{D}\|.
	\end{eqnarray*}
	Therefore, $\Theta$ is a surjective contractive linear isomorphism. So by the open mapping theorem, $\Theta^{-1}$ exists and is a bounded linear surjective isomorphism. 
	
	Fix $T \in L(\mathcal{D}, Y)$ and let $T_1$ and $T_2$ be any two extensions of $T$ in $L(F_Y(X), Y)$. If $\gamma = \sum\limits_{n=1}^{\infty} a_{n} \delta_{x_{n}}^{Y} \in \mathcal{D}$, then 
	$$\gamma \left({T_{1}} \circ \delta_{X}^{Y} - {T_{2}} \circ \delta_{X}^{Y}\right) = {T_{1}}\left(\sum\limits_{n=1}^{\infty} a_{n} \delta_{x_{n}}^{Y}\right) - {T_{2}}\left(\sum\limits_{n=1}^{\infty} a_{n} \delta_{x_{n}}^{Y}\right)=T(\gamma)-T(\gamma)=0.$$ Thus ${T_{1}} \circ \delta_{X}^{Y} + \ ^{\diamondsuit}\mathcal{D} = {T_{2}} \circ \delta_{X}^{Y} + \ ^{\diamondsuit}\mathcal{D}$ and hence $\widetilde{T} \circ \delta_{X}^{Y}+ \ ^{\diamondsuit}\mathcal{D}$ is uniquely determined by $T$. Now it follows that $\Theta^{-1}(T)=\widetilde{T} \circ \delta_{X}^{Y}+ \ ^{\diamondsuit}\mathcal{D}$ for all $T \in L(\mathcal{D}, Y)$, where $\widetilde{T} \in L(F_{Y}(X), Y)$ is a norm preserving extension of $T$. 	Since 
	$$\|\Theta^{-1}(T)\|=\|\widetilde{T} \circ \delta_{X}^{Y}+ \ ^{\diamondsuit}\mathcal{D}\|\leq Lip(\widetilde{T} \circ \delta_{X}^{Y}) \leq \|T\|,$$ 
	$\Theta^{-1}$ is also a contraction. Now, it is easy to conclude that $\Theta$ is surjective linear isometry so that $\bigslant{Lip_{0}(X, Y)}{^{\diamondsuit}\mathcal{D} }$ is isometrically isomorphic to $L(\mathcal{D}, Y)$.
\end{proof}
\begin{rem} \label{lip qo lin}
	Let $X$ be a Banach space and $Y$ be an injective Banach space. Then it follows from Proposition \ref{quotient of lip} and Lemma \ref{eqn2} that the quotient space $\bigslant{Lip_{0}(X, Y)}{L(X, Y)}$ is isometrically isomorphic to $L(\ker(\beta_X^Y), Y)$.
\end{rem}
\begin{prop} \label{prop q.o free space}
    Let $\mathcal{D}$ be a closed subspace of $F_{Y}(X)$, then $L\left(\bigslant{F_Y(X)}{\mathcal{D} }, Y\right) $ is linear isometrically isomorphic to $  \ ^{\diamondsuit}\mathcal{D}$.
\end{prop}
\begin{proof}
	For a fix $f \in~^{\diamondsuit}\mathcal{D}$, we define $\Lambda_f : F_Y(X) \xrightarrow{} Y$ given by $\Lambda_f(\gamma)=\gamma(f)$ for all $\gamma \in F_Y(X)$. Then easily we can verify that $\Lambda_f$ is a linear map with $\|\Lambda_f\| \leq Lip(f)$. Further
	$\ker(\Lambda_f) = \left\{ \gamma \in F_Y(X): \gamma(f)=0\right\} \supset \mathcal{D}$. Therefore $\Lambda_f$ induces a bounded linear map (again denoted by $\Lambda_f$) from $\bigslant{F_Y(X)}{\mathcal{D} }$ into $ Y$. This induces a map 
	$$\Lambda : \ ^{\diamondsuit}\mathcal{D} \xrightarrow{} L\left(\bigslant{F_Y(X)}{\mathcal{D} }, Y\right)$$ 
	given by $\Lambda(f)(\gamma+\mathcal{D})=\gamma(f)$ for all $f \in~^{\diamondsuit}\mathcal{D}$ and $\gamma \in F_Y(X)$. Then $\Lambda$ is linear. Also 
	\begin{eqnarray*}
		\|\Lambda(f)(\gamma+\mathcal{D})\| &=& \|\gamma(f)\| \\ 
		&=& \inf\limits_{\mu \in \mathcal{D}}\|(\gamma+ \mu)(f)\| \\ 
		&\le& \inf\limits_{\mu \in \mathcal{D}}\|(\gamma+ \mu)\| Lip(f) \\ 
		&=& Lip(f)\|\gamma + \mathcal{D}\|,
	\end{eqnarray*}
	$$$$
	for all $f \in~^{\diamondsuit}\mathcal{D}$ and $\gamma \in F_Y(X)$. Thus $\Lambda$ is a contraction. We show that $\Lambda$ is an isometry. 
	
	Let $\epsilon >0$. Then there exist $x_{1},x_{2} \in X$ with $x_{1}\neq x_{2}$ such that $$\frac{\|f(x_{1})-f(x_{2})\|}{\|x_{1}-x_{2}\|} \geq Lip(f)-\epsilon, \ \mbox{that is} \ \left\|\frac{\delta_{x_{1}}^{Y}-\delta_{x_{2}}^{Y}}{\|x_{1}-x_{2}\|}(f)\right\|\geq Lip(f)-\epsilon.$$ Therefore, $$\left\|\Lambda(f)\left(\frac{\delta_{x_{1}}^{Y}-\delta_{x_{2}}^{Y}}{\|x_{1}-x_{2}\|} + \mathcal{D}\right)\right\|=\frac{\|f(x_{1})-f(x_{2})\|}{\|x_{1}-x_{2}\|} \geq Lip(f)-\epsilon.$$ 
	Since $1=\left\|\frac{\delta_{x_{1}}^{Y}-\delta_{x_{2}}^{Y}}{\|x_{1}-x_{2}\|}\right\|\geq \left\|\frac{\delta_{x_{1}}^{Y}-\delta_{x_{2}}^{Y}}{\|x_{1}-x_{2}\|} + \mathcal{D}\right\|$ and $\epsilon > 0$ is arbitrary, we conclude that $\Vert \Lambda(f) \Vert \ge Lip(f)$. Hence $\Lambda$ is an isometry.
	
	To prove the surjectivity of $\Lambda$, let $T \in L\left(\bigslant{F_Y(X)}{\mathcal{D} }, Y\right)$. Set $f(x) = T(\delta_{x}^{Y}+\mathcal{D})$ for all $x \in X$. Then for $x_1, x_2 \in X$, we have 
	$$\Vert f(x_1) - f(x_2) \Vert = \Vert T(\delta_{x_1}^Y - \delta_{x_2}^Y + \mathcal{D}) \Vert \le \Vert T \Vert \Vert \delta_{x_1}^Y - \delta_{x_2}^Y + \mathcal{D} \Vert \le \Vert T \Vert \Vert x_1 - x_2 \Vert.$$
	Thus $f \in  Lip_{0}(X, Y)$. Now for $\gamma = \sum\limits_{n=1}^{\infty} a_{n}  \delta_{x_{n}}^{Y} \in F_Y(X)$, 
	\begin{eqnarray*}
		\gamma(f) &=& \left(\sum\limits_{n=1}^{\infty} a_{n}  \delta_{x_{n}}^{Y}\right)(f) \\
		&=& \sum\limits_{n=1}^{\infty} a_{n}  f(x_{n}) \\
		&=& \sum\limits_{n=1}^{\infty} a_{n}  T(\delta_{x_{n}}^{Y}+\mathcal{D}) \\
		&=& T\left(\sum\limits_{n=1}^{\infty} a_{n}  \delta_{x_{n}}^{Y} +\mathcal{D}\right) \\ 
		&=& T(\gamma+\mathcal{D}).
	\end{eqnarray*}
	Therefore for all $\gamma \in \mathcal{D}, \ \ \gamma(f)=0$ so that $f \in \ ^{\diamondsuit}\mathcal{D}$. This completes the proof. 
\end{proof}
\begin{rem}
\begin{enumerate}
    \item  In particular for $\mathcal{D}=\{0\}$ we have the Theorem \ref{thm 1}.
    \item For $\mathcal{D} =\ker(\beta_X^Y)$ we get $L\left(\bigslant{F_Y(X)}{\ker(\beta_X^Y) }, Y\right)$ is linear isometrically isomorphic to $  L(X, Y)$.
\end{enumerate}  
\end{rem}
\begin{rem} \label{note3.19}
\begin{enumerate}
  \item  If we consider $Y$ to be an injective dual space and $\mathcal{D} =\ker(\beta_X^Y)$, then from the Remark \ref{ker p with quotient}, Proposition \ref{quotient of lip} and Lemma \ref{eqn2} we have $ \ker(P_{X}^{Y}) $ is topologically isomorphic to $ \bigslant{Lip_{0}(X, Y)}{L(X, Y) }$which is linear isometrically isomorphic to $ L(\ker(\beta_{X}^{Y}), Y)$.
  
  In particular for $ X = Y=\mathbb{R}$ we avail 
$\left(\ker(\beta_{\mathbb{R}}^{\mathbb{R}})\right)^{\ast}$ is linear isometrically isomorphic to $\bigslant{L^{\infty}(\mathbb{R})}{\mathbb{R}}$. ( It is known that $Lip_0(\mathbb{R}, \mathbb{R})$ is isometrically isomorphic to $L^{\infty}(\mathbb{R})$ (see \cite[p 128]{LFBS})).
    \end{enumerate}
    \end{rem}

\section{An example}

    The following proposition answers that the space $L^{\infty}(\mathbb{R})$ quotient by the set of all constant functions, that is $span\{\textbf{1}\}$; is a dual space, in fact, we provide its explicit predual space. In this section, by $\mathcal{S}$ we denote the space of all simple functions on $\mathbb{R}$.
\begin{prop}
    $\ker(\beta_{\mathbb{R}}^{\mathbb{R}})$ is isometrically isomorphic to $\{f \in L^{1}(\mathbb{R}) : \int_{\mathbb{R}}^{} f  d\mu=0\}$, where $\mu$ is the Lebesgue measure on $\mathbb{R}$.
\end{prop}
\begin{proof}
    Let us define $\Lambda :  L^1(\mathbb{R}) \xrightarrow{} \mathbb{R}$ given by 
    $$ \Lambda(f) = \int_{\mathbb{R}} f d\mu.$$
    Then $\Lambda$ is a surjective linear contraction.

    Again recall (for details see \cite[p 542]{IROLFSOCDIDS}) the linear isometric isomorphism between $F(\mathbb{R})$ and $L^1(\mathbb{R})$ say
    $\phi:F(\mathbb{R}) \xrightarrow{} L^1(\mathbb{R})$ whose action on the spanning elements is as follows: 
    \[
		\phi(\delta_x^{\mathbb{R}}) =
		\begin{cases}
        - \chi_{(x,0)} &\text{;if } x < 0\\
        0 &\text{;if } x = 0 \\
			\chi_{(0,x)} &\text{;if } x > 0. 
		\end{cases}
		\]
         In fact, from the proof it follows that $\mathcal{R}_{\mathbb{R}}= span\{\delta_x^{\mathbb{R}}: x \in \mathbb{R}\}$ is linearly isometrically isomorphic to the space of all simple functions $ \mathcal{S}$ $\left(\subset L^1(\mathbb{R})\right)$.
         
 Let $x \in \mathbb{R}$. Then 
 \[
		\Lambda \circ \phi (\delta_x^{\mathbb{R}}) =
		\begin{cases}
        \Lambda(- \chi_{(x,0)})=x=\beta_{\mathbb{R}}^{\mathbb{R}}(\delta_x^{\mathbb{R}}) &\text{;if } x < 0\\
        0 &\text{;if } x = 0 \\
			\Lambda(\chi_{(0,x)})=x =\beta_{\mathbb{R}}^{\mathbb{R}}(\delta_x^{\mathbb{R}})&\text{;if } x > 0. 
		\end{cases}
		\]
  Thus $\beta_{\mathbb{R}}^{\mathbb{R}} = \Lambda \circ \phi$.

 Further 
 \begin{eqnarray*}
     \ker(\beta_{\mathbb{R}}^{\mathbb{R}}) &=& \left\{\gamma \in F(\mathbb{R}): \beta_{\mathbb{R}}^{\mathbb{R}}(\gamma)=0\right\}\\
     &\cong& \left\{f \in L^1(\mathbb{R}): \beta_{\mathbb{R}}^{\mathbb{R}}\left(\phi^{-1}(f)\right)=0\right\}\\
     &=& \left\{f \in L^1(\mathbb{R}): \Lambda(f)=0\right\}\\
     &=&\{f \in L^{1}(\mathbb{R}) : \int_{\mathbb{R}}^{} f  d\mu=0\},
 \end{eqnarray*}
 where the symbol $'\cong'$ means linearly isometrically isomorphic.
 This completes the proof.
\end{proof}
\begin{cor}
   $\ker(\beta_{\mathbb{R}}^{\mathbb{R}})\cap \mathcal{R}_{\mathbb{R}}$ is linearly isometrically isomorphic to $ \{f \in \mathcal{S}: \int_{\mathbb{R}}^{} f  d\mu=0\}$, where $\mathcal{R}_{\mathbb{R}}= span\left\{\delta_x^{\mathbb{R}}: x \in X\right\}$.
\end{cor}
\begin{rem}
\begin{enumerate}
     \item From the above Corollary we have $\overline{\ker(\beta_{\mathbb{R}}^{\mathbb{R}})\cap \mathcal{R}_{\mathbb{R}}}^{\|.\|}$ is linear isometrically isomorphic to $ \overline{\{f \in \mathcal{S}: \int_{\mathbb{R}}^{} f  d\mu=0\}}^{\|\cdot\|_1}$. Further suppose $(f_n)$ be a sequence in $\ker(\beta_{\mathbb{R}}^{\mathbb{R}})\cap \mathcal{R}_{\mathbb{R}}$ such that $f_n \xrightarrow{\|.\|_1} f \in L^1(\mathbb{R})$. Then
    \begin{eqnarray*}
        \left|\int_{\mathbb{R}}^{} f  d\mu\right|&=& \left|\int_{\mathbb{R}}^{} (f-f_n + f_n)  d\mu\right|\\
        &\leq& \int_{\mathbb{R}}^{} |f-f_n|  d\mu + \left|\int_{\mathbb{R}}^{} f_n  d\mu\right|.
    \end{eqnarray*}
Thus $\int_{\mathbb{R}}^{} f  d\mu = 0$ so that    
    $\overline{\{f \in \mathcal{S}: \int_{\mathbb{R}}^{} f  d\mu=0\}}^{\|.\|_1} = \{f \in L^{1}(\mathbb{R}): \int_{\mathbb{R}}^{} f  d\mu=0\}$. Therefore, $\overline{\ker(\beta_{\mathbb{R}}^{\mathbb{R}})\cap \mathcal{R}_{\mathbb{R}}}^{\|.\|} = \ker(\beta_{\mathbb{R}}^{\mathbb{R}})$.    
    Furthermore, Corollary \ref{lip qo lin} allows us to further deduce that $\ker(\beta_{\mathbb{R}}^{\mathbb{R}})^{\ast}$ is linear isometrically isomorphic to $ \bigslant{Lip_{0}(\mathbb{R}, \mathbb{R})}{L(\mathbb{R}, \mathbb{R}) } $ which is further linear isometrically isomorphic to $\bigslant{L^{\infty}(\mathbb{R})}{span\{\textbf{1}\}}$.

\item Further if we consider $X=Y=\mathbb{R}$, $\mathcal{D}= \ker(\beta_{\mathbb{R}}^{\mathbb{R}})$, closed subspace of $F_{\mathbb{R}}(\mathbb{R})$; then the Proposition \ref{prop q.o free space} provides that 
    $L\left(\bigslant{F_{\mathbb{R}}(\mathbb{R})}{\ker(\beta_{\mathbb{R}}^{\mathbb{R}}) }, \mathbb{R}\right) $ is linear isometrically isomorphic to $  \ ^{\diamondsuit}\ker(\beta_{\mathbb{R}}^{\mathbb{R}}) = L(\mathbb{R}, \mathbb{R})$, where $L(\mathbb{R},\mathbb{R})$ is identified as a subspace of $L^{\infty}(\mathbb{R},\mathbb{R})$ consisting of all constant functions.
    \end{enumerate}
\end{rem}
{\bf Declarations and acknowledgements.} 
The second author was financially supported by the Senior Research Fellowship from National Institute of Science Education and Research Bhubaneswar funded by Department of Atomic Energy, Government of India. The authors have no competing interests or conflict of interest to declare that are relevant to the content of this article.
\bibliographystyle{plain}

\end{document}